\documentclass[12pt]{amsproc}

\usepackage{fullpage}
\usepackage{amsmath}
\usepackage{amsfonts}
\usepackage{amssymb}
\usepackage{amsthm}
\usepackage[utf8]{inputenc}
\usepackage{breqn}
\usepackage{afterpage}
\usepackage{longtable}
\usepackage{indentfirst}
\usepackage{caption}
\usepackage{subcaption}
\usepackage{graphicx}
\graphicspath{ {./images/} }
\newtheorem{theorem}{Theorem}[section]
\newtheorem{lemma}[theorem]{Lemma}

\theoremstyle{definition}
\newtheorem{definition}[theorem]{Definition}

\newtheorem{theorem-definition}[theorem]{Theorem-Definition}

\theoremstyle{remark}
\newtheorem{remark}[theorem]{Remark}

\numberwithin{equation}{section}

\usepackage{color}
\usepackage{xcolor}

\newcommand{\C}{\mathbb{C}}

\begin{document}

\title{Capacity dimension of the Brjuno set in  ${\C}^n$}
\begin{author}[N.~Akramov]{Nurali Akramov}
    \address{National University of Uzbekistan,  Tashkent, Uzbekistan}
\email{nurali.akramov.1996$@$gmail.com}
\end{author}
\begin{author}[K.~Rakhimov]{Karim Rakhimov}

\address{V.I. Romanovskiy Institute of Mathematics of Uzbekistan Academy of Sciences,  Tashkent, Uzbekistan}
\email{karimjon1705$@$gmail.com}
\end{author}

\date{ }

\begin{abstract} In this work, we prove that the complement of the Brjuno set in $\mathbb{C}^n$ has zero $C_\sigma$-capacity with respect to the kernel $k_\sigma(z,\xi)=\|z-\xi\|^{-2n+2}|\log{\|z-\xi\||^{\sigma}}$ for any $\sigma>n$. In particular, it follows that it has  zero $h_\delta$-Hausdorff measure  with respect to the $h_\delta(t)=t^{2n-2}|\log{t}|^{-\delta}$, for any $\delta>n+1$. This generalizes a previous result of Sadullaev and the second author in dimension one to higher dimensions.

\end{abstract}

\maketitle

\section{Introduction}

The linearization of holomorphic germs near a fixed point is a central topic in the study of local holomorphic dynamics (see, for example, \cite{MAB}, \cite{FBR}, \cite{JRA2} and \cite{JRE}). Consider a holomorphic germ $f:(\mathbb{C}^n,0)\to(\mathbb{C}^n,0)$ that has a fixed point at the origin
\begin{equation}\label{eq:eq1}
  f(z)=\Lambda z+P_2(z)+...+P_d(z)+\dots,  
\end{equation}
where the linear part is given by the diagonal matrix $\Lambda=Df(0)=\text{diag}(\lambda_1,\lambda_2,...,\lambda_n)$ and $P_d(z):\mathbb{C}^n\to\mathbb{C}^n$ is a homogeneous polynomial of degree $d\ge2$. The fundamental question is whether $f$ is \textit{linearizable}, i.e., holomorphically conjugate to $\Lambda z$. Formally, we seek a holomorphic map $\varphi$, invertible in a neighborhood of the origin, such that
$$\varphi^{-1}(z)\circ{f(z)}\circ\varphi(z)=\Lambda{z}.$$ 

The Brjuno condition, introduced by A. Brjuno, provides a sharp criterion for linearizability in such cases (see \cite{ABR}).  Let us define the Brjuno condition. For $\lambda:=(\lambda_1, \lambda_2, ...,\lambda_n)\in\mathbb{C}^n$, and for an integer $m\ge{2}$, define
\begin{equation}\label{eq:OMEGA}
   \Omega(\lambda, m):=\min\{|\lambda^k-\lambda_j|: 2\le|k|\le{m},  
   1\le{j}\le{n}\}, 
\end{equation}
where $k=(k_1,k_2,...,k_n)\in\mathbb{N}^n$, $|k|=k_1+k_2+...+k_n$, and $\lambda^k=\lambda_1^{k_1}\lambda_2^{k_2}\lambda_3^{k_3}\dots\lambda_n^{k_n}$.

\begin{definition} \label{Brjuno1}
    Let $\lambda\in\mathbb{C}^n$  and $  \Omega(\lambda, m)$ be as in \eqref{eq:OMEGA}. The vector $\lambda$ is said to satisfy the \textit{Brjuno condition} if
    \begin{equation}\label{eq:BCO}
        \sum_{j=1}^\infty\frac{1}{2^j}\log{\frac{1}{\Omega(\lambda,2^j)}}<\infty.
    \end{equation}
\end{definition}
When $\Omega(\lambda, m)=0$ for some $m$ we say that $\lambda$ is \textit{resonant}, that is,  $\lambda$ is called resonant if there exists a multi-index $k=(k_1,k_2,\dots,k_n)\in\mathbb{N}^n$ such that  
$$\lambda^k-\lambda_j=\lambda_1^{k_1}\lambda_2^{k_2}\cdots\lambda_n^{k_n}-\lambda_j=0$$ 
for some $1\leq j\leq n$.
\begin{theorem}[{Brjuno \cite{ABR}}]
  Assume that  $\lambda\in\mathbb C^n$ is not resonant.  If $\lambda$ satisfies the Brjuno condition, then the holomorphic germ \eqref{eq:eq1} is holomorphically linearizable. 
\end{theorem}

In dimension 1, Yoccoz (see \cite{JCY}) showed that if $\lambda$ does not satisfy the Brjuno condition, then $f(z) = \lambda z + z^2$ is not holomorphically linearizable at the origin. However, in higher dimensions, it is not clear whether this remains true (see \cite{JRA1}).
 
In this paper, we are interested in the capacity dimension of the complement of the Brjuno set, i.e. the set which does not satisfy the Brjuno condition. For the case $n = 1$, A. Sadullaev and the second author  proved the following result (see \cite{AS1}). 
\begin{theorem}[{Sadullaev-Rakhimov, \cite{AS1}}]\label{t:sadrakh}
    The set of points in $\mathbb{C}$ that do not satisfy the Brjuno condition \eqref{eq:BCO} has  zero capacity with respect to the kernel $k_\sigma(z,\xi)=|\log{|z-\xi|}|^{\sigma}$,  $z,\xi\in\mathbb C,$ for any $\sigma>2$. 
 \end{theorem}
For definitions and notation, we refer the reader to Section \ref{refs2}. Our main result is a generalization of Theorem \ref{t:sadrakh} for $n\ge2$. We use $\|\cdot\|$ to denote the Euclidean distance in $\mathbb{C}^n$.
\begin{theorem}\label{t:main}
    Let $n\ge2$ and $E$ be the  set of points in $\mathbb{C}^n$ which do not satisfy the Brjuno condition \eqref{eq:BCO}. Then $E$ has  zero capacity     with respect to the kernel 
    \begin{equation}
\label{eq:kzeta}
k_\sigma(z,\xi)=\frac{|\log{\|z-\xi\|}|^{\sigma}}{\|z-\xi\|^{2n-2}},\,z,\xi\in\mathbb C^n,  
    \end{equation} for any $\sigma>n$. In particular, $E$ has  zero $h_\delta$-Hausdorff measure  with respect to the $h_\delta(t)=t^{2n-2}|\log{t}|^{-\delta}$, for any $\delta>n+1$.  
 \end{theorem}
When $n = 1$, A. Sadullaev and the second author (see \cite{AS1}) used a number-theoretic approach (see Section \ref{PL}) to prove Theorem \ref{t:sadrakh}. However, in higher dimensions, such a direct number-theoretic approach is not available. While it might be tempting to assume that Theorem \ref{t:sadrakh} can be extended inductively to higher dimensions, this is not always the case. Even if two complex numbers $\lambda_1$ and $\lambda_2$ individually satisfy the Brjuno condition, their pair $\lambda = (\lambda_1, \lambda_2)$ may fail to belong to the Brjuno set, for instance, if the product $\lambda_1 \lambda_2$ does not satisfy the condition.

The paper is organized as follows.  In Section \ref{refs2} we recall the definitions and some properties of Haussdorff $h_{\delta}$-measure and $C_\sigma$-capacity.  In Section \ref{refs1} we define the Brjuno condition in a different context and study some properties. Finally, in Section \ref{s:proofmain} we prove our main result.
\subsection*{Acknowledgment}  The authors would like to thank V.I. Romanovskiy Institute of Mathematics of the Academy of Sciences of the Republic of Uzbekistan and National University of Uzbekistan for the warm welcome and the excellent work conditions. 

\section{Hausdorff measure and capacity} \label{refs2}
\subsection{$h$-Hausdorff measure}

Let $h:[0,r_0]\to [0,+\infty)$ be a strictly increasing continuous function with $h(0)=0$ and $r_0>0$. Let $E\subset\mathbb{R}^{n}$ be a bounded set and fix positive $\varepsilon$ with $\varepsilon<r_0$. Consider a cover of $E$ by a finite collection of open balls $\{B_j(x_j, r_j)\}_{j=1}^m$ such that $r_j < \varepsilon$ for all $1 \le j \le m$, where $m$ depends on the chosen cover. Define 
$$H^{h}(E,\varepsilon)=\inf\left\{\sum_{j=1}^m h(r_j):\text{ } \bigcup_{j=1}^m{B_j}\supset E \right\}.$$
It is clear that $H^{h}(E,\varepsilon)$ is an increasing function of $\varepsilon$. Then, the limit
$$H^{h}(E)=\lim_{\varepsilon\to0+}H^{h}(E,\varepsilon)$$
exists and is called the \textit{$h$-Hausdorff measure} of $E$.  When $h_{\alpha}(t)=t^\alpha,$  $\alpha>0$,  the measure $H^{h_\alpha}(E)$ is known as the classic $\alpha$-Hausdorff measure of $E$.

\subsection{$C_\sigma$-capacity} $\mathbb C^n$-capacities are one of the important tools in pluripotential theory. In particular, their null sets are pluripolar sets, which vanish $t^{2n-2}|\log t |^{-\delta}$-Haussdorf measure  for any $\delta>1$. Many researchers, including A. Sadullaev, E. Bedford, and B.A. Taylor, have made significant contributions in this area (see  \cite{ERB},\cite{AS1},\cite{AS2}). In this section, we define a $\mathbb C^n$-capacity as in \cite{NSL}, such that its null set is slightly larger than pluripolar sets.

Let $K\subset\mathbb{C}^n$ be compact and $$k_\sigma(z,\xi)=\frac{|\log||z-\xi|||^\sigma}{||z-\xi||^{2n-2}},\,\,z,\xi\in\mathbb{C}^n,$$
where  $\sigma>0$. Denote by $\mathring{M}_K^+$ the set of positive probability measures $\mu$, with $|\mu| = 1$, supported in $K$.
The following integral 
$$U^\mu(z)=\int_K k_\sigma(z,\xi)d\mu(\xi)$$
is called the \textit{potential} of measure $\mu\in \mathring{M}_K^{+}$. 
 Let
$$I(\mu)=\int\limits_K U^\mu(z)d\mu(z)$$
and $W(K)=\inf\{I(\mu):\mu\in \mathring{M}_K^+\}$. Then \textit{$C_\sigma$-capacity} of $K$ is defined as
$$C_\sigma(K):=\frac{1}{W(K)}.$$
For an arbitrary set $E\subset\mathbb{C}^n$ the inner capacity is defined as $$\underline{C}_\sigma(E)=\sup_{K\subset{E}}C_\sigma(K)$$ and the outer capacity as $\overline{C}_\sigma(E)=\inf_{G\supset{E}}\underline{C}_\sigma(G)$ where $G$ is an open set. The classic properties of $C_\sigma$ capacity from the general theory of capacities (see \cite{LCARL}, \cite{NSL}).
\begin{enumerate}
    \item For every Borel set $E\subset\mathbb{C}^n$: $\overline{C}_\sigma(E)=\underline{C}_\sigma(E)=C_\sigma(E)$.
    \item The capacity $C_\sigma(E)=0$, if and only if there exists a finite Borel measure $\mu\in\overset{\circ}  {M^+_E}$ such that $U^\mu(z)\equiv+\infty$.
    \item  If $n\ge 2$ and $C_\sigma(E)=0$, then the Hausdorff $h_{\delta}$-measure  of $E$ with respect to the gauge function $h(t)=t^{2n-2}|\log t|^{-\delta}$ is zero for any $\delta>\sigma+1$ (see \cite{NSL}).
    \item  For the sequence of compact sets $\{K_j\}_{j=1}^{\infty}$, the capacity satisfies:
    $$C_\sigma\bigg(\bigcup_{j=1}^\infty K_j\bigg)\le\sum_{j=1}^{\infty}C_\sigma(K_j).$$
    \item Proper analytic subsets of $\mathbb{C}^n$ have zero $C_\sigma$-capacity. 
    \item Let $U, V\subset \mathbb{C}^n$ be open sets and $\phi:U\to V$ be a conformal map. If $C_\sigma(E)=0$ for $E\subset U$, then $C_\sigma(\phi(E))=0.$
    \end{enumerate}
The following technical lemma will be needed later.
 \begin{lemma}\label{l:intineq}
     Let $\sigma$, $a$ and $\eta$ be positive numbers satisfying $a<\eta<\frac{1}{2}$ and $n\ge 2$ be an integer. Then there exist positive $A_1,A_2$ and $B_1,B_2$ independent of $a$ such that we have
     $$A_1 |\log a|^{\sigma+1}-A_2\le\int_0^\eta\frac{|\log (r^2+a)|^\sigma}{(r^2+a)^{n-1}}r^{2n-3}dr\le B_1 |\log a|^{\sigma+1}+B_2.$$
 \end{lemma}
\begin{proof}
Assume first $n=2$. Then we have
    \begin{align*}
       \int_0^\eta\frac{|\log (r^2+a)|^\sigma}{r^2+a}rdr &=  \frac{1}{2}\int_0^\eta{(-\log (r^2+a))^\sigma}d\log (r^2+a)\\
       &=-\frac{1}{2(\sigma+1)}{(-\log (r^2+a))^{\sigma+1}}|_0^\eta\\
       &=\frac{1}{2(\sigma+1)}{|\log a|^{\sigma+1}}-\frac{1}{2(\sigma+1)}{|\log (a+\eta^2)|^{\sigma+1}}.
    \end{align*}
    So in this case we take $A_1=B_1=\frac{1}{2(\sigma+1)}$ and $A_2=\frac{1}{2(\sigma+1)}{|\log \eta|^{\sigma+1}}$ and $B_2=0.$ 
    
Assume now $n\ge 3$. Denote $t=r^2+a$. Then we have 

\begin{align*}
    \int_0^\eta\frac{|\log (r^2+a)|^\sigma}{(r^2+a)^{n-1}}r^{2n-3}dr&=\frac12\int_a^{\eta^2+a}\frac{|\log t|^\sigma}{t^{n-1}}(t-a)^{n-2}dt\\
    &=\frac12\int_a^{\eta^2+a}\frac{|\log t|^\sigma}{t}dt+\frac{1}{2}\sum_{j=2}^{n-2}(-1)^ja^jC_{n-2}^j\int_a^{\eta^2+a}\frac{|\log t|^\sigma}{t^j}dt\\
    &\ge \frac{1}{2(\sigma+1)}(|\log a|^{\sigma+1}-|\log (a+\eta^2)|^{\sigma+1})\\
    &-\frac{1}{2}\sum_{j=2}^{n-2}a^jC_{n-2}^j|\log a|^\sigma\int_a^{\eta^2+a}\frac{1}{t^{j}}dt\ge A_1 |\log a|^{\sigma+1}-A_2  
\end{align*}
for some positive $A_1,A_2$ independent of $a$. On the other hand, we have

\begin{align*}
    \int_0^\eta\frac{|\log (r^2+a)|^\sigma}{(r^2+a)^{n-1}}r^{2n-3}dr&=\frac12\int_a^{\eta^2+a}\frac{|\log t|^\sigma}{t^{n-1}}(t-a)^{n-2}dt\\
    &\le \frac12\int_a^{\eta^2+a}\frac{|\log t|^\sigma}{t}dt\\
    &\le B_1 |\log a|^{\sigma+1}+B_2  
\end{align*}
for some $B_1,B_2$ independent of $a$.
\end{proof}

\section{Brjuno condition}\label{refs1}

In this section, we introduce the necessary definitions and concepts related to the Brjuno condition. In particular, we define the Brjuno condition in a different, yet equivalent, context. Denote the set of integer vectors $$\mathbb{N}_j=\{k=(k_1,k_2,\dots,k_n)\in\mathbb{Z}^n: k_i\ge 0,i\ne j, k_j\ge -1 \}$$ where one coordinate $k_j$ is permitted to take values not less than $-1$, while all others are non-negative. Define $\mathbb{N}_0$ as the union of these sets
$\mathbb{N}_0=\cup_{j=1}^n \mathbb N_j.$ For $k=(k_1,k_2,...,k_n)\in\mathbb{N}_0$ denote $|k|=k_1+k_2+...+k_n$ and $kz=k_1z_1+...+k_nz_n$ where $z=(z_1,z_2,...,z_n)\in \mathbb{C}^n$. By $B(z,r)$ we denote the ball with radius $r>0$ and centered at $z$.

Let $z\in\mathbb{C}^n$. For an integer $m\ge{2}$, define
\begin{equation}\label{eq:omega}
\omega(z, m):=\min\{|kz-p|: 1\le|k|\le{m},k\in\mathbb{N}_0, p\in\mathbb{Z}\}.  
\end{equation}
Note that $kz-p=k_1z_1+...+k_nz_n-p$ is a scalar.

\begin{definition}\label{Brjuno2}
    Let $z\in\mathbb{C}^n$ and $\omega(z, m)$ be as defined in \eqref{eq:omega}. We say that $z$  satisfies the \textit{Brjuno condition} (with respect to $\omega(z, m)$) if
    
    \begin{equation}\label{eq:BC}
       \sum_{j=1}^\infty\frac{1}{2^j}\log{\frac{1}{\omega(z, 2^j)}}<\infty.
    \end{equation}
\end{definition}

We denote by $\mathcal{B}_n$ the Brjuno set, defined as the set of all points $z\in\mathbb{C}^n$ satisfying the Brjuno condition \eqref{eq:BC}.
\begin{remark}\label{re:equivBrj}
 Although Definitions \ref{Brjuno1} and \ref{Brjuno2} are not formally equivalent—due to the difference between expressions \eqref{eq:OMEGA} and \eqref{eq:omega}—we can see that they are actually related. Let $\lambda = (\lambda_1, \lambda_2, \dots, \lambda_n) \in \mathbb{C}^n$ with $\lambda_1 \cdots \lambda_k \neq 0$. Then it is straightforward to verify that $\lambda$ satisfies the Brjuno condition according to Definition \ref{Brjuno1} if and only if$$z=(z_1,\dots,z_n)=\left(\frac{1}{2\pi{i}}\log{\lambda_1}, \dots,\frac{1}{2\pi{i}}\log{\lambda_n}\right)$$ satisfies the Brjuno condition in the sense of Definition \ref{Brjuno2}.  Indeed, the assertion follows by the following elementary fact: when $r\to 0$ and $\|\alpha\|_{\mathbb{Z}}$ are small enough, there exists $c_1,c_2>0$ independent of $\alpha$ and $r$ such that  $$c_1|\|\alpha\|_{\mathbb{Z}}+ir|^2\le |e^{2\pi r}e^{2\pi i\alpha}-1|^2\le c_2 |\|\alpha\|_{\mathbb{Z}}+ir|^2, $$
where $\|\alpha\|_{\mathbb{Z}}$  is the distance from $\alpha$ to $\mathbb{Z}$.
\end{remark}

\subsection{Dimension 1.}{\label{PL}} It is clear that when $n=1$, all non-real numbers satisfy the Brjuno condition \eqref{eq:BC}.  For $\alpha\in\mathbb{R}$, there is a number theoretical approach to the Brjuno condition. Namely, if $\alpha$ is a rational number, then it clearly does not satisfy \eqref{eq:BC}. If $\alpha$ is an irrational number, then we can write it as 
$$\alpha=\frac{1}{a_1+\frac{1}{a_2+\frac{1}{a_3+\cdots}}}=:[a_1,a_2,a_3,...].$$
 A finite part $[a_1,a_2,a_3,...,a_j]=\frac{P_j}{Q_j}$ of the continued fraction becomes a rational number. Moreover, $\{\frac{P_j}{Q_j}\}$ is the fastest convergence sequence to $\alpha$. Then $\alpha$ satisfies Brjuno condition (see \cite{ABR}) if and only if
 $$\sum_{j=1}^\infty\frac{\log Q_{j+1}}{Q_j}<+\infty.$$

Let us review the main steps of the proof of Theorem \ref{t:sadrakh}.
If $\alpha\in\mathbb{R}\setminus\mathbb{Q}$ does not satisfy  Brjuno condition then for any $\varepsilon>0$ we have (see \cite{AS1})
\begin{equation}\label{eq:qn=infty}
\sum_{j=1}^\infty\frac{\log^{2+\varepsilon} Q_{j+1}}{Q_j^{2+\frac{\varepsilon}{4}}}=+\infty. 
\end{equation}
Using \eqref{eq:qn=infty} and some other properties of continued fractions Sadullaev and the second author proved that the potential
$$ U(z)=\int |\log |z-p/q||^{2+\varepsilon} d\mu = \sum_{q=2}^\infty\sum_{p=1}^{q-1}\frac{|\log |z-p/q||^{2+\varepsilon}}{q^{2+\frac{\varepsilon}{4}}}$$
diverges when $z\in \mathcal B_1\cap [0,1]$ where
\begin{equation}\label{eq:mun=1}
\mu:=\sum_{q=2}^\infty\sum_{p=1}^{q-1}\frac{\delta_{\frac{p}{q}}}{q^{2+\frac{\varepsilon}{4}}}. 
\end{equation}
As usual, $\delta_a$ denotes the Dirac measure at $a$. Then the property (2) of $C_\sigma$-capacity implies that $C_\sigma$-capacity of the complement of $\mathcal{B}_1$ vanishes. So Theorem \ref{t:sadrakh} follows.

\subsection{Preparatory lemmas} To prove our main result, we need a replacement for \eqref{eq:qn=infty} and \eqref{eq:mun=1} in higher dimensions. In this subsection we prove a result  that plays the role of   \eqref{eq:qn=infty} in dimension $n\ge 2$. Take $z\in \mathbb{C}^n$ with $n\ge 2$. It is clear that if $\omega(z,2^j)$ is uniformly bounded from below by a positive constant, then \eqref{eq:BC} holds and $z\in\mathcal{B}_n$. 

\begin{lemma}\label{l:sequence}
   Let $z\in \mathbb{C}^n$. Assume $\omega(z,2^j)\to 0$ as $j\to \infty$. Then there exists a strictly increasing sequence of positive integers $\{j_m\}$ with $j_1=1$ such that
   \begin{equation}\label{eq:sequence}
       \omega(z,2^{j_m})= \omega(z,2^{j_m+1})=...=\omega(z,2^{j_{m+1}-1})>\omega(z,2^{j_{m+1}}).
   \end{equation}
Moreover, \eqref{eq:BC} holds if and only if
\begin{equation}\label{eq:BCjk}
     \sum_{m=1}^\infty\frac{1}{2^{j_m}}\log{\frac{1}{\omega(\lambda, 2^{j_m})}}<\infty.
\end{equation}

\end{lemma}
\begin{proof}
Since $\omega(z,2^j)\to0$ as $j\to\infty$, it is clear that there is a unique sequence $\{j_m\}$ satisfying the above condition. In order to prove the second conclusion, it is easy to clarify
\begin{align*}
    \sum_{m=1}^\infty\frac{1}{2^{j_m}}\log{\frac{1}{\omega(\lambda, 2^{j_m})}}&\le  \sum_{j=1}^\infty\frac{1}{2^j}\log{\frac{1}{\omega(\lambda, 2^j)}}\\ &=  \sum_{m=1}^\infty\frac{1}{2^{j_m}}\sum_{j=j_m}^{j_{m+1}-1}\frac{1}{2^{j-j_m}}\log{\frac{1}{\omega(\lambda, 2^{j_m})}}\\
    &\le 2 \sum_{m=1}^\infty\frac{1}{2^{j_m}}\log{\frac{1}{\omega(\lambda, 2^{j_m})}}.
\end{align*}
       Hence, \eqref{eq:BCjk} holds if and only if \eqref{eq:BC} holds.
\end{proof}

The following lemma serves as a replacement for \eqref{eq:qn=infty} in dimension $n\ge2$. 
\begin{lemma}\label{l:bruno=+infty}
    Let $K\subset \mathbb{C}^n$ be a compact set. Then there exists $l\in \mathbb{N}$ depending only on $K$ such that if $z\in K$ does not satisfy the Brjuno condition {\eqref{eq:BC}}, then we have
    \begin{equation}\label{eq:n+1+varep=+}
        \sum_{|k|=1,\text{ } k\in\mathbb{N}_0}^\infty\sum_{|p|=0}^{l(|k|+1)}\frac{\left|\log |kz-p|^2\right|^{n+1+\varepsilon}}{|k|^{n+1+\frac{\varepsilon}{4}}}=+\infty,
    \end{equation}
  for any $\varepsilon>0$. 
\end{lemma}
\begin{proof} Since $z$ does not satisfy the Brjuno condition \eqref{eq:BC}, we have $\omega(z,2^j)\to 0$ as $j\to \infty$. Then by Lemma  \ref{l:sequence}  there exists a sequence $\{j_m\}$ satisfying \eqref{eq:sequence}. Moreover, we have 
$$\sum_{m=1}^\infty\frac{1}{2^{j_m}}\log{\frac{1}{\omega(\lambda, 2^{j_m})}}=\infty$$
Take $0<\delta<1$. By applying H\"older's inequality,  we obtain the following inequality:
$$\sum_{m=1}^\infty\frac{1}{2^{j_m}}\bigg|\log{\frac{1}{\omega(z,2^{j_m})}}\bigg|\le\bigg(\sum_{m=1}^\infty\frac{1}{2^{j_m(n+1+\varepsilon)(1-\delta)}}\bigg|\log{\frac{1}{\omega(z,2^{j_m})}}\bigg|^{n+1+\varepsilon}\bigg)^{\frac{1}{n+1+\varepsilon}}\bigg(\sum_{m=1}^{\infty}\frac{1}{2^{j_m\frac{n+1+\varepsilon}{n+\varepsilon}\delta}}\bigg)^{\frac{n+\varepsilon}{n+1+\varepsilon}}.$$

It is well known that the second series on the right-hand side converges for any $\delta>0$. Let $\delta>0$ be sufficiently small so that $(n+1+\varepsilon)(1-\delta)\ge n+1+\frac{\varepsilon}{4}$.
Then, we have
\begin{align*}
\sum_{m=1}^{\infty}\frac{1}{2^{j_m(n+1+\frac{\varepsilon}{4})}}\bigg|\log\frac{1}{\omega(z,2^{j_m})}\bigg|^{n+1+\varepsilon}&\ge  \sum_{m=1}^\infty\frac{1}{2^{j_m(n+1+\varepsilon)(1-\delta)}}\bigg|\log{\frac{1}{\omega(z,2^{j_m})}}\bigg|^{n+1+\varepsilon}\\
&\ge C \sum_{m=1}^\infty\frac{1}{2^{j_m}}\bigg|\log{\frac{1}{\omega(z,2^{j_m})}}\bigg|=+\infty,
\end{align*}
where
$C=\bigg(\sum_{m=1}^{\infty}\frac{1}{2^{j_m\frac{n+1+\varepsilon}{n+\varepsilon}\delta}}\bigg)^{-\frac{n+\varepsilon}{n+1+\varepsilon}}$.

Take $l\in \mathbb{N}$, with $l>2$ satisfying $K\subset B(0,\frac{l}{2n})$. It is clear that for $|p| >l(|k|+1)$ we have $|kz-p|\ge 1$. Indeed, if  $|p| >l(|k|+1)$, since $z\in B(0,\frac{l}{2n})$ we have 
\begin{equation}\label{eq:kzp-1}
    |kz|= |k_1z_1+k_2z_2+...+k_nz_n|\le (|k_1|+|k_2|+...+|k_n|)\frac{l}{2}\le (|k|+1)l\le |p|-1.
\end{equation}
Hence,  if $\omega(z,2^{j_m})= |k_1^mz_1+k_2^mz_2+...+k_n^mz_n-p_m|$ for some $(k^m,p_m)\in \mathbb{N}_0\times\mathbb{Z}$ then since $\omega(z,2^{j_m})<1$ we have $|p_m|\le l(|k^m|+1)$. Moreover, since $\omega(z,2^{j_m})>\omega(z,2^{j_{m+1}})$, we have $(k^m,p_m)\ne (k^{\tilde m},p_{\tilde m})$ for $m\ne \tilde m$.
Consequently, since $|k^m|\le 2^{j_m}$ we have
\begin{align*}
\sum_{|k|=1,\text{ } k\in\mathbb{N}_0}^\infty\sum_{|p|=0}^{l(|k|+1)}\frac{\left|\log |kz-p|^2\right|^{n+1+\varepsilon}}{|k|^{n+1+\frac{\varepsilon}{4}}}&=\sum_{m=1}^{\infty}\frac{1}{|k_m|^{n+1+\frac{\varepsilon}{4}}}\bigg|\log|k^mz-p_m|^2\bigg|^{n+1+\varepsilon}+\\
&+\sum_{|k|=1,\text{ }k\not=k^m,\text{ } k\in\mathbb{N}_0}^\infty\sum_{p=0}^{l(|k|+1)}\frac{\left|\log |kz-p|^2\right|^{n+1+\varepsilon}}{|k|^{n+1+\frac{\varepsilon}{4}}}\\
&\ge \sum_{m=1}^{\infty}\frac{1}{2^{j_m(n+1+\frac{\varepsilon}{4})}}\bigg|\log|k^mz-p_m|^2\bigg|^{n+1+\varepsilon}=+\infty.\end{align*}
Hence, we have \eqref{eq:n+1+varep=+}.
\end{proof}

\section{Proof of the main result}\label{s:proofmain}

We recall that ${\mathcal{B}}_n$ denotes the complex numbers $z=(z_1,z_2,\dots,z_n)\in \mathbb{C}^n$ satisfying the Brjuno condition \eqref{eq:BC}.
\begin{theorem}
   For any $\sigma>n$, we have $C_\sigma(C{\mathcal{B}}_n)=0$,  where $C{\mathcal{B}}_n$ is the complement of ${\mathcal{B}}_n$ in $\mathbb{C}^n$. 
\end{theorem}
 \begin{proof} Note that $C_\sigma$-capacity vanishes on proper analytic subsets of $\mathbb{C}^n$,
 hence
 \begin{equation}\label{eq:Pi}
     \Pi:=\left\{w=(w_1,w_2,...,w_n)\in\mathbb{C}^n: \prod_{j=1}^n w_j=0\right\}
 \end{equation}
has zero $C_\sigma$-capacity. Since $C_\sigma$-capacity is countably sub-additive and $\Pi$ has zero $C_\sigma$-capacity,   it is enough to show that $C_\sigma(V\cap C{\mathcal{B}}_n)=0$ for any bounded open set $V$ satisfying
 \begin{equation}\label{eq:condition}
     \overline{V}\cap\Pi=\emptyset.
 \end{equation} 
 Fix a bounded open set $V$ satisfying \eqref{eq:condition}. Let $l\ge 1$ be an integer as in Lemma \ref{l:bruno=+infty} with $K=\overline{V}$. 
 Define a measure $\mu$ as follows 
$$\mu:=\sum_{|k|=1,\text{ } k\in\mathbb{N}_0}^\infty\sum_{|p|=0}^{l(|k|+1)}\frac{\mu_{k,p}}{|k|^{n+1+\frac{\varepsilon}{4}}}$$
where for $k\in\mathbb{N}_0$ and $p\in \mathbb{Z}$  the measure $\mu_{k,p}$ is the natural extension of the Lebesgue measure $\mathrm{Leb}_{k,p}$ on $kw=p$ to $\mathbb{C}^n$, i.e. for any bounded Borel set $E\subset \mathbb{C}^n$ we have $ \mu_{k,p}(E):=\mathrm{Leb}_{k,p}(E\cap \{kw=p\}) $.
Consider the restriction of $\mu$ to the set $\overline{V}$, denoted by $\tilde{\mu}:=\mu|_{\overline{V}}$.   We claim that $\tilde{\mu}$ is finite. Indeed, it is not difficult to see that there exists a constant $C>0$ independent of $k$ and $p$ such that $\mu_{k,p}(\overline{V})\le C$. Thus, we obtain
\begin{align*}
 \tilde{\mu}(\mathbb{C}^n)=\tilde{\mu}(\overline{V})\le C\sum_{|k|=1,\text{ } k\in\mathbb{N}_0}^\infty\sum_{|p|=0}^{l(|k|+1)}\frac{1}{|k|^{n+1+\frac{\varepsilon}{4}}}&\le (2l+1) C\sum_{|k|=1,\text{ } k\in\mathbb{N}_0}^\infty\frac{1}{|k|^{n+\frac{\varepsilon}{4}}}\\
 &\le \tilde{C}\sum_{j=1}^\infty\frac{1}{j^{1+\frac{\varepsilon}{4}}}<+\infty,   
\end{align*} 
where $\tilde C$ is a positive constant.

Next, we analyze the potential $U^{\tilde{\mu}}(z)$, which is given by

$$U^{\tilde{\mu}}(z)=\sum_{|k|=1,\text{ } k\in\mathbb{N}_0}^\infty\sum_{|p|=0}^{l(|k|+1)}\frac{1}{|k|^{n+1+\frac{\varepsilon}{4}}}\int_{\bar{V}}\frac{|\log||w-z|||^{{ n +\varepsilon}}}{||w-z||^{2n-2}}d\mu_{k,p}(w).$$
It is clear that if $z$ is far from $V$, then we have $U^{\tilde{\mu}}(z)<+\infty$. Next, we shall show that  it is $\infty$ on $V\cap C{\mathcal{B}}_n$.

\textbf{Claim.} \textit{We have $U^{\tilde{\mu}}(z)=+\infty$ for any $z\in V\cap C{\mathcal{B}}_n$.}

Assuming the claim, we will finish the proof. By applying the claim together with the second property of the  $C_\sigma$-capacity, we conclude that $C_\sigma(V\cap C{\mathcal{B}}_n)=0$. Hence, it remains to prove the claim.
\begin{proof}[{Proof of the claim}] Fix $z\in V\cap C{\mathcal{B}}_n$ and  $0<\eta<1/10$ with $B(z,3n\eta)\Subset V$. Define
$$\mathcal{N}_\eta=\left\{(k,p)\in \mathbb{N}_0\times \mathbb{Z}:|kz-p|<{\eta}\right\}.$$
Since $z\in V\cap C{\mathcal{B}}_n$, it is clear that $\mathcal{N}_\eta\ne \emptyset$  for any $\eta>0$. Similarly as in \eqref{eq:kzp-1} we can show that  for $(k,p)\in\mathcal N_\eta$ we have $|p|\le l(|k|+1)$.

Fix $(k,p)\in\mathcal N_\eta$. 
It is well known that the closest point  $\tilde{z}$ to $z$ in $\Pi_{k,p}=\{w\in \mathbb C^n: kw=p\}$ is 
$$\tilde{z}=z+\frac{p-kz}{\|k\|^2}k,$$
where $\|k\|=\sqrt{k_1^2+k_2^2+k_3^2+...+k_n^2}$.
It is not difficult to see that $\|\tilde{z}-z\|\le \frac{\eta}{2}$ and hence $B(\tilde z,\eta)\Subset V$. Without loss of generality assume that $k_n=\max_{1\le{i}\le{n}} k_i$. Let's make the following unitary linear substitution $L:\mathbb{C}^n\to \mathbb{C}^n $ as follows
\begin{align*}
   w_j&=\tilde w_j+\tilde z_j, \,\ 1\le j\le n-1,\\
   w_n&=\tilde w_n+\tilde z_n-\sum_{j=1}^{n-1} \frac{k_j}{k_n}\tilde w_j.
\end{align*}
Then
\begin{align*}
kw-p&=\sum\limits_{j=1}^n k_j w_j-p=\sum\limits_{j=1}^{n}k_j\tilde z_j+k_n\tilde w_n-p=\\
&=\sum\limits_{j=1}^nk_j\left(z_j+\frac{p-kz}{\|k\|^2}k_j\right)+k_n\tilde{w}_n-p=k_n\tilde{w}_n.\end{align*}
So, we have $$L(\Pi_{k,p})=\{\tilde{w}=(\tilde w_1,...,\tilde w_n)\in\mathbb{C}^n:\tilde w_n=0\}.$$
Since $L$ is a translation we have $L^*(\mathrm{Leb}_{k,p})=\mathrm{Leb}_{n-1}(\tilde w')$, where $\mathrm{Leb}_{n-1}(\tilde{w}')$ is the Lebesgue measure in $\mathbb{C}^{n-1}\times\{0\}$ and $\tilde{w}'=(\tilde{w}_1,...,\tilde w_{n-1})$. 

Let us now show that for $w\in\Pi_{k,p}$ and $(\tilde w',0)=L^{-1}(w)$ we have 
\begin{equation}\label{eq:ineq}
    \|w-z\|^2\le 4(\|\tilde w'\|^2+|kz-p|^2),
\end{equation}
where $\|\tilde w'\|^2=|\tilde w_1|^2+...+|\tilde w_{n-1}|^2$. Indeed,
 for $1\le j\le n-1$ we obtain the following
\begin{align*}
  \left| w_j-z_j\right|^2&=\left|\tilde{w}_j+\tilde{z}_j-z_j\right|^2=\left|\tilde{w}_j+\frac{k_j}{\|k\|^2}(p-k_jz_j)\right|^2\\
  & \le2\left(|\tilde{w}_j|^2+\left|\frac{k_j}{\|k\|^2}(p-k_jz_j)\right|^2\right).
\end{align*}
Similarly, since $\tilde w_n=0$ we have
 \begin{align*}
     |w_n-z_n|^2&=\left|\tilde{w}_n+\tilde{z}_n-\sum_{j=1}^{n-1}\frac{k_j}{k_n}\tilde w_j-z_n\right|^2\\
     &=\left|\tilde{w}_n+\frac{k_n}{\|k\|^2}(p-k_nz_n)-\sum_{j=1}^{n-1}\frac{k_j}{k_n}\tilde w_j\right|^2  \\
     &\le2\left|\frac{k_n}{\|k\|^2}(p-k_nz_n)\right|^2+2\sum_{j=1}^{n-1}\frac{k_j^2}{k_n^2}\left|\tilde{w}_j\right|^2.
     \end{align*}
Thus, we conclude 
\begin{align*}
   \|w-z\|^2&=|w_1-z_1|^2+|w_2-z_2|^2+...+|w_{n-1}-z_{n-1}|^2+|w_n-z_n|^2\\
   &\le2\sum_{j=1}^{n-1}\left(|\tilde{w_j}|^2+\left|\frac{k_j}{\|k\|^2}(p-k_jz_j)\right|^2\right)+2\left|\frac{k_n}{\|k\|^2}(p-k_nz_n)\right|^2+
   2\sum_{j=1}^{n-1}\frac{k_j^2}{k_n^2}\left|\tilde{w_j}\right|^2\\
   &=2\sum_{j=1}^{n-1}\left(1+\frac{k_j^2}{k_n^2}\right)|\tilde{w_j}|^2+2\sum_{j=1}^{n}\left|\frac{k_j}{\|k\|^2}(p-k_jz_j)\right|^2\\
   & \le 4(\|\tilde w'\|^2+|kz-p|^2),
\end{align*}
where in the last step we used $|k_j|\le k_n$. Consequently, we obtain \eqref{eq:ineq}. Moreover, thanks to \eqref{eq:ineq} and  since $|kz-p|<\eta$  we have 
\begin{equation}\label{eq:subset}
L(B(0,\eta))\subset B(z,3\eta)\Subset V. 
\end{equation}

Let us now show that $U^{\tilde{\mu}}(z)=+\infty$.  Note that $\frac{|\log|r||^{{ n +\varepsilon}}}{r^{2n-2}}$ is  decreasing for  $0<r<3\eta$. Thanks to \eqref{eq:ineq} and \eqref{eq:subset} there exists a constant $C_1$ independent of $k,p$ such that
 \begin{align*}
    \int_{\overline{V}}\frac{|\log||w-z|||^{{ n +\varepsilon}}}{||w-z||^{2n-2}}d\mu_{k,p}(w)&=\int_{\overline{V}\cap \Pi_{k,p}}\frac{|\log||w-z|||^{{ n +\varepsilon}}}{||w-z||^{2n-2}}d\,\mathrm{Leb}_{k,p}(w)\\
    &=\int_{L^{-1}(\overline{V}\cap \Pi_{k,p})}\frac{|\log||L^{-1}(w)-z|||^{{ n +\varepsilon}}}{||L^{-1}(w)-z||^{2n-2}}L^*(d\,\mathrm{Leb}_{k,p}(w))\\
    &\ge
    C_1\int_{|\tilde {w}'|<\eta}\frac{\left|\log\left(\|\tilde{w}'\|^2+\|kz-p\|^2\right) \right|^{{ n +\varepsilon}}}{\left(\|\tilde{w}'\|^2+\left\|kz-p\right\|^2\right)^{n-1}}d\,\mathrm{Leb}_{n-1}(w').
 \end{align*}
After going to spherical coordinates we obtain 
\begin{align*}
\int_{|\tilde{w}'|<\eta}\frac{\left|\log\left(\|\tilde{w}'\|^2+\|kz-p\|^2\right) \right|^{{n +\varepsilon}}}{\left(\|\tilde{w}'\|^2+\left\|kz-p\right\|^2\right)^{n-1}}d\,\mathrm{Leb}_{n-1}(w')&= C_2 \int_0^\eta\frac{\log^{{ n +\varepsilon}}{(r^2+\|kz-p\|^2)}}{(r^2+\|kz-p\|^2)^{n-1}}r^{2n-3}dr
\end{align*}
for some $C_2>0$ independent of $k,p$ and $z$.
Thanks to Lemma \ref{l:intineq}  there exist positive constants $A_1,A_2$ independent of $k,p$ and $z$  such that 

\begin{align*}
    \int_0^\eta \frac{\log^{{ n +\varepsilon}}{(r^2+\|kz-p\|^2)}}{(r^2+\|kz-p\|^2)^{n-1}}r^{2n-3}dr&\ge A_1{|\log \|kz-p\|^2|^{n+1+\varepsilon}}-A_2.
\end{align*}

Finally we have,
\begin{equation}\label{eq:lastineq} \int_{\overline{V}}\frac{|\log||w-z|||^{{ n +\varepsilon}}}{||w-z||^{2n-2}}d\mu_{k,p}(w)\ge C_1C_2A_1{|\log \|kz-p\|^2|^{n+1+\varepsilon}}-C_1C_2A_2.\end{equation}
Put $C_3:=C_1C_2A_1,$ and
$$A_3:=\sum_{|k|=1,\text{ } k\in\mathbb{N}_0}^\infty\sum_{|p|=0}^{l(|k|+1)}\frac{1}{|k|^{n+1+\frac{\varepsilon}{4}}}<\infty,$$ 
$$A_4:=\sum_{|k|=1,\text{ } k\in\mathbb{N}_0}^\infty\sum_{|p|=0}^{l(|k|+1)}\frac{\log^{n+1+\varepsilon}\left((|k|+1)^2\max_{w\in\overline{V}}\|w\|^2+|p|+1\right)}{|k|^{n+1+\frac{\varepsilon}{4}}}<\infty.$$ 
 Note that  for $(k,p)\notin\mathcal{N}_\eta$ we have $|kz-p|\ge\eta$ and hence $$|\log{\|kz-p\|^2}|\le \log\left((|k|+1)^2\max_{w\in\overline{V}}\|w\|^2+|p|+1\right)+ |\log{\eta^2}|.$$  Then, by the last inequality and thanks to \eqref{eq:lastineq} and the fact that $(k,p)\in\mathcal N_\eta$ implies $|p|\le l(|k|+1)$, we obtain
\begin{align*}
    U^{\tilde{\mu}}(z)+C_3(A_3|\log{\eta^2}|^{n+1+\varepsilon}+A_4)
   & \ge C_1C_2A_1\sum_{|k|=1,\text{ } k\in\mathbb{N}_0}^\infty\sum_{|p|=0}^{l(|k|+1)}\frac{1}{|k|^{n+1+\frac{\varepsilon}{4}}}|\log{|kz-p|^2}|^{n+1+\varepsilon}-C_1C_2A_2A_3.
\end{align*}
So there are positive constants $C_4,C_5$ such that  
$$U^{\tilde{\mu}}(z)\ge C_4\sum_{|k|=1,\text{ } k\in\mathbb{N}_0}^\infty\sum_{|p|=0}^{l(|k|+1)}\frac{1}{|k|^{n+1+\frac{\varepsilon}{4}}}|\log{|kz-p|^2}|^{n+1+\varepsilon}-C_5.$$
Thanks to Lemma \ref{l:bruno=+infty}  we have 
$$\sum_{|k|=1,\text{ } k\in\mathbb{N}_0}^\infty\sum_{|p|=0}^{l(|k|+1)}\frac{1}{|k|^{n+1+\frac{\varepsilon}{4}}}|\log{\|kz-p\|^2}|^{n+1+\varepsilon}=+\infty$$
and hence we have  $U^{\tilde{\mu}}(z)=+\infty$. 
 
\end{proof}   

\end{proof}
We now complete proving our main result.
\begin{proof}[{Proof of Theorem \ref{t:main}}]   Fix simply connected open set $B\subset \mathbb{C}^n$ such that, $\overline{B}\cap  \Pi=\emptyset$ and $$\psi(w)=\left(\frac{1}{2\pi{i}}\log{w_1}, \dots,\frac{1}{2\pi{i}}\log{w_n}\right)$$  defines a  conformal map on $B$, where $\Pi$ is defined as \eqref{eq:Pi}.  Thanks to Remark \ref{re:equivBrj}, $\lambda\in E\cap B$ if and only if $\psi(\lambda)\in C\mathcal{B}_n$.
By Theorem 4.1, we have $C_\sigma(C\mathcal{B}_n)=0$ for any $\sigma>n$. Consequently, by  propery 6 of  $C_\sigma$-capacity  we have $C_\sigma(E\cap B)=0$ for any for any $\sigma>n$. Since  $\psi$ is locally conform outside $\Pi$, and  $\Pi$ has zero $C_\sigma$-capacity (i.e. $C_\sigma(\Pi)=0$) and  that $C_\sigma$-capacity is countably sub-additive it follows that $C_\sigma(E)=0$, for any $\sigma>n$. Thanks to property 3 of $C_\sigma$-capacity the second assertion follows.
\end{proof}


\begin{thebibliography}{9}
    \bibitem{MAB} M. Abate, \emph{Discrete holomorphic local dynamical systems}. In “Holomorphic Dynamical Systems”,Eds. G.Gentili, J. Guénot, G. Patrizio, Lect. Notes in Math. 1998, Springer, Berlin, 2010, pp. 1–55.
    \bibitem{ERB} E. Bedford, B. A. Taylor, \emph{A new capacity for plurisubharmonic functions}, Acta Mathematica 149 (1982), 1-40.
    \bibitem{FBR} F. Bracci, \emph{Local dynamics of holomorphic diﬀeomorphisms}, Boll. UMI (8), 7–B (2004), pp. 609–636.
    \bibitem{ABR} A. D. Brjuno, \emph{Analytical form of differential equations}, Trans. Moscow Math. Soc. Volume 25 (1971), 131-288.
    \bibitem{LCARL} L. Carleson, \emph{Selected Problems on Exceptional Sets, Van Nostrand Mathematical Studies, vol. 13}, D. Van Nostrand Co., Inc., Princeton, N.J.-Toronto, Ont.-London, 1967.
    \bibitem{NSL} N. S. Landkof, \emph{Capacity and Hausdorﬀ measure. Estimates for potentials}, Uspekhi Mat. Nauk, 1965, Volume 20, Issue 2, 189–195.
    \bibitem{JRA1} J. Raissy, \emph{Brjuno conditions for linearization in presence of resonances}, Asymptotics in Dynamics, Geometry and PDEs; Generalized Borel Summation Volume I, November (2009), 201–218.
    \bibitem{JRA2}J. Raissy: \emph{Holomorphic linearization of commuting germs of holomorphic maps}, J. Geom. Anal., Vol 23, (2013), no. 4, pp. 1993-2019.
    \bibitem{RKA} K. Rakhimov, \emph{ Capacity dimension of Perez-Marco set},  Contemporary Mathematics, Volume 662 (2016), 131-138.  
    \bibitem{JRE} J. Reppekus, \emph{Small divisors in discrete local holomorphic dynamics}.  Bollettino dell'Unione Matematica Italiana Volume 16 (2023), 173–202. 
    \bibitem{AS1} A. Sadullaev, K. Rakhimov, \emph{Capacity Dimension of the Brjuno Set}, Indiana University Mathematics Journal Volume 64, No. 6 (2015), 1829-1834.
    \bibitem {AS2} A. Sadullaev, \emph{Plurisubharmonic measures and capacities on complex manifolds.} Uspekhi mat. nauk, T.36: 4, (1981), 53-105 (in Russian).
    \bibitem{JCY} J.-C. Yoccoz, \emph{Th\'eor\`eme de Siegel, nombres de Brjuno et polyn\^omes quadratiques: Petits diviseurs en dimension 1}, Ast\'erisque 231 (1995), 3–88 (French).
     
   
     \end{thebibliography}
\end{document}